\documentclass[10pt,twoside]{article}
\usepackage[utf8]{inputenc}

\title{\uppercase{\textbf{\large{On the motivic cohomology of some singular rings}}}}

\author{\textsc{tess bouis}}
\date{}
\usepackage[left=3cm, right=3cm, top=1in, bottom=1in, headheight=2cm]{geometry}

\usepackage{fancyhdr}
\usepackage{stmaryrd}

\usepackage{mathdots}

\usepackage{amsmath, amscd, amsfonts, amssymb, amsthm,todonotes,eurosym,enumitem,relsize}

\usepackage[utf8]{inputenc}
\usepackage{dirtytalk}
\usepackage[T1]{fontenc}
\usepackage{mathrsfs}

\usepackage[colorlinks]{hyperref}
\usepackage[figure,table]{hypcap}
\definecolor{imperialred}{RGB}{237, 41, 57}
\definecolor{royalblue}{RGB}{64, 106, 212}
\definecolor{link}{RGB}{11,0,128}
\definecolor{gren}{RGB}{32,130,63}
\hypersetup{
	bookmarksnumbered,
	pdfstartview={FitH},
	citecolor=royalblue,
	linkcolor=imperialred,
	urlcolor=link,
	linktocpage = true
	pdfpagemode={UseOutlines}
}

\usepackage{xcolor}
\usepackage{comment}

\usepackage[object=pgfhan]{pgfornament}

\usepackage{graphicx}
\usepackage{array}
\usepackage{csquotes}
\usepackage{extarrows}

\usepackage{needspace}

\input{xy}
\xyoption{all}
\usepackage{tikz-cd}
\usepackage{textcomp}
\usepackage{colonequals}

\usepackage{sectsty}
\sectionfont{\sc\centering}
\subsectionfont{\bf\large}
\subsubsectionfont{\bf\normalsize}

\usepackage{fancyhdr}

\newlength{\outermargin} 
\newlength{\mar} 
\newlength{\len}
\newlength{\temp}
\newtheorem{theorem}{Theorem}[section]
\newtheorem{lemma}[theorem]{Lemma}
\newtheorem{corollary}[theorem]{Corollary}
\newtheorem{proposition}[theorem]{Proposition}

\theoremstyle{definition}
\newtheorem{remark}[theorem]{Remark}

\newtheorem{example}[theorem]{Example}

\DeclareMathOperator{\Z}{\mathbb{Z}}
\DeclareMathOperator{\Q}{\mathbb{Q}}

\DeclareMathOperator{\F}{\mathbb{F}}

\DeclareMathOperator{\C}{\mathbb{C}}

\renewcommand{\epsilon}{\varepsilon}

\DeclareFontFamily{U}{MnSymbolC}{}
\DeclareFontShape{U}{MnSymbolC}{m}{n}{
	<-5.5> MnSymbolC5
	<5.5-6.5> MnSymbolC6
	<6.5-7.5> MnSymbolC7
	<7.5-8.5> MnSymbolC8
	<8.5-9.5> MnSymbolC9
	<9.5-11.5> MnSymbolC10
	<11.5-> MnSymbolCb12
}{}

\usepackage[bbgreekl]{mathbbol}
\DeclareSymbolFontAlphabet{\mathbb}{AMSb}
\DeclareSymbolFontAlphabet{\mathbbl}{bbold}

\usepackage{soul}
\usepackage{enumitem}
\numberwithin{equation}{theorem}

\begin{document}
	
	\maketitle
	
	\pagestyle{fancy}
	\fancyhead[EC]{TESS BOUIS}
	\fancyhead[OC]{ON THE MOTIVIC COHOMOLOGY OF SOME SINGULAR RINGS}
	\fancyfoot[C]{\thepage}
	
	\begin{abstract}
            Using non-$\mathbb{A}^1$-invariant motivic cohomology, we prove motivic refinements of certain known computations of the algebraic $K$-theory of singular rings, such as rings of the form $\Z/p^n$, $\Z[x]/(x^e)$, and $\mathscr{C}(X;\C)$ for $X$ a compact Hausdorff space. These refinements are made possible by the use of integral $p$-adic Hodge theory, as a replacement for the standard use of trace methods in $K$-theory.
	\end{abstract}

	{
		\hypersetup{linkcolor=black}
		\tableofcontents
	}

    \section{Introduction}

    \vspace{-\parindent}
    \hspace{\parindent}

    Motivic cohomology, as envisioned by Beilinson and Lichtenbaum \cite{lichtenbaum_values_1973,lichtenbaum_values_1984,beilinson_notes_1986,beilinson_height_1987,beilinson_notes_1987}, is a cohomological refinement of the algebraic $K$-theory of schemes. It was first developed, at the initiative of Bloch and Voevodsky \cite{bloch_algebraic_1986,voevodsky_cycles_2000}, in the generality of smooth schemes over a field or a Dedekind domain \cite{voevodsky_cycles_2000,levine_techniques_2001,geisser_motivic_2004,spitzweck_commutative_2018,bachmann_very_2025}. This approach relied crucially on $\mathbb{A}^1$\nobreakdash-invariant techniques, which are permitted by Quillen's theorem that algebraic $K$-theory is $\mathbb{A}^1$\nobreakdash-invariant on regular schemes \cite{quillen_higher_1973}. For not necessarily regular schemes, a more general definition of motivic cohomology, based on trace methods in algebraic $K$-theory, was recently introduced in the works of Elmanto--Morrow \cite{elmanto_motivic_2023} (for equicharacteristic schemes) and of the author \cite{bouis_motivic_2024,bouis_weibel_2025} (in general). This definition is non-$\mathbb{A}^1$-invariant in general, although it uses the classical $\mathbb{A}^1$-invariant motivic cohomology of smooth schemes as an input. Moreover, this new definition indeed generalises the classical definition in terms of motivic $\mathbb{A}^1$\nobreakdash-homotopy theory, as it recovers it for smooth schemes over a field \cite{elmanto_motivic_2023} or over a Dedekind domain~\cite{bouis_beilinson-lichtenbaum_2025}. Note that an alternative approach to non-$\mathbb{A}^1$-invariant motivic cohomology was also recently introduced by Kelly--Saito in terms of the pro cdh topology \cite{kelly_procdh_2024}; the fact that these two definitions agree is \cite[Corollary~C]{bouis_weibel_2025}.

    The aim of this article is to prove that this new, non-$\mathbb{A}^1$-invariant theory of motivic cohomology can actually be computed, or at least concretely manipulated, in several classes of interesting examples. We focus on five classes of commutative rings for which motivic cohomology was previously not defined:
    \begin{enumerate}
        \item finite chain rings, {\it e.g.}, $\Z/p^n$ (see Section~\ref{Sectionfinitechain});
        \item truncated polynomials, {\it e.g.}, $\Z[x]/(x^e)$ (see Section~\ref{Sectiontruncated});
        \item perfect and semiperfect rings, {\it e.g.}, $\F_p[x^{1/p^{\infty}},y^{1/p^{\infty}}]/(x-y)$ (see Section~\ref{Sectionperfect});
        \item henselian valuation rings, {\it e.g.}, $\mathcal{O}_{\C_p}$ (see Section~\ref{Sectionvaluation});
        \item $\C^\star$-algebras, {\it e.g.}, the algebra of continuous $\C$-valued functions on a compact Hausdorff space (see Section~\ref{SectionCstar}).
    \end{enumerate}
    
    These five classes of examples are treated independently. The algebraic $K$-theory of these commutative rings was already studied by several authors and for different purposes. We often use these existing results to establish, via the Adams decomposition of rational algebraic $K$-theory, the rational part of our results (except in Section~\ref{Sectiontruncated}, where we use de Rham cohomology to reprove the known $K$-theoretic results). Establishing our results integrally is often more subtle, and relies crucially on the approach to motivic cohomology taken in \cite{elmanto_motivic_2023,bouis_motivic_2024,bouis_weibel_2025}, {\it i.e.}, on the use of trace methods and, in mixed characteristic, on the integral $p$-adic Hodge theory of Bhatt--Morrow--Scholze \cite{bhatt_topological_2019} and Bhatt--Scholze \cite{bhatt_prisms_2022}. 
    
    Note finally that while some of our results are mere motivic refinements of known results in algebraic $K$-theory, passing to the motivic level can also yield more understanding of the existing $K$-theoretic results (see for instance Theorems~\ref{theoremtruncatedpolynomialmain} and~\ref{theoremvaluationringsmotiviccohomologyfinitecoefficients}, Corollary~\ref{corollarytruncatedpolynomialintegral}, and Remarks~\ref{remarknilpotenceofv1} and \ref{remarknegativeKgroupsperfectalgebras}).

    \subsection*{Acknowledgements.}

    I would like to thank Elden Elmanto and Matthew Morrow for helpful discussions, Matthew Morrow for comments on a draft of this paper, and the referee for helpful comments and corrections. This project has received funding from the European Research Council (ERC) under the European Union’s Horizon 2020 research and innovation programme (grant agreement No. 101001474).

    \vspace{-\parindent}
    \hspace{\parindent}

	\section{Finite chain rings}\label{Sectionfinitechain}
	
	\vspace{-\parindent}
	\hspace{\parindent}
	
	Finite chain rings are commutative rings $\mathcal{O}_K/\pi^n$, where $\mathcal{O}_K$ is a mixed characteristic discrete valuation ring with finite residue field, $\pi$ is a uniformizer of $\mathcal{O}_K$, and $n \geq 1$ is an integer. Examples of finite chain rings thus include finite fields, rings of the form $\Z/p^n$, and truncated polynomials over a finite field.
	
	\begin{lemma}\label{lemmafinitechainringsdecomposition}
		Let $\mathcal{O}_K$ be a discrete valuation ring of mixed characteristic $(0,p)$ and with finite residue field $\F_q$, $\pi$ be a uniformizer of $\mathcal{O}_K$, and $n \geq 1$ be an integer. Then for every integer $i \geq 0$, there is a natural equivalence
		$$\Z(i)^{\emph{mot}}(\mathcal{O}_K/\pi^n) \simeq \left\{
        \begin{array}{ll}
			\Z[0] & \emph{if } i=0\\
			\Z_p(i)^{\emph{syn}}(\mathcal{O}_K/\pi^n) \oplus \Z(i)^{\emph{mot}}(\F_q)[\tfrac{1}{p}] & \emph{if } i \geq 1
		\end{array}
		\right.$$
		in the derived category $\mathcal{D}(\Z)$.
	\end{lemma}
	
	\begin{proof}
		The result for $i=0$ follows from the equivalences
		$$\Z(0)^{\text{mot}}(\mathcal{O}_K/\pi^n) \simeq R\Gamma_{\text{cdh}}(\mathcal{O}_K/\pi^n,\Z) \simeq R\Gamma_{\text{cdh}}(\F_q,\Z) \simeq \Z[0]$$
		in the derived category $\mathcal{D}(\Z)$, the first equivalence being \cite[Example~$4.68$]{bouis_motivic_2024}, the second equivalence being nilpotent invariance of cdh sheaves, and the last equivalence being a consequence of the fact that fields are local for the cdh topology.
		
		For every integer $i \geq 0$, the commutative diagram
		$$\begin{tikzcd}
			\Z(i)^{\text{mot}}(\mathcal{O}_K/\pi^n) \ar[r] \ar[d] & \Z_p(i)^{\text{syn}}(\mathcal{O}_K/\pi^n) \ar[d] \\
			\Z(i)^{\text{mot}}(\F_q) \ar[r] & \Z_p(i)^{\text{syn}}(\F_q)
		\end{tikzcd}$$
		is a cartesian square in the derived category $\mathcal{D}(\Z)$ \cite[Proposition~$3.29$]{bouis_motivic_2024}. If $i \geq 1$, the bottom right term vanishes (use for instance the description of Bhatt--Morrow--Scholze's syntomic cohomology in characteristic $p$ in terms of logarithmic de Rham--Witt forms), and there is a natural equivalence
		$$\Z(i)^{\text{mot}}(\F_q) \xlongrightarrow{\sim} \Z(i)^{\text{mot}}(\F_q)[\tfrac{1}{p}]$$
		in the derived category $\mathcal{D}(\Z)$ (by a classical result in motivic cohomology, see also Theorem~\ref{theoremmotiviccohomologyofperfectschemes}\,$(2)$ for a more general statement), hence the desired result.
	\end{proof}
	
	\begin{proposition}
		Let $\mathcal{O}_K$ be a mixed characteristic discrete valuation ring with finite residue field, $\pi$ be a uniformizer of~$\mathcal{O}_K$, and $n \geq 1$ be an integer. Then for every integer $m \in \Z$, there is a natural isomorphism
		$$\emph{K}_m(\mathcal{O}_K/\pi^n) \cong \left\{
		\begin{array}{llll}
			\Z & \emph{if } m=0\\
			\emph{H}^1_{\emph{mot}}(\mathcal{O}_K/\pi^n,\Z(i)) & \emph{if } m=2i-1, \emph{ } i \geq 1\\
			\emph{H}^2_{\emph{mot}}(\mathcal{O}_K/\pi^n,\Z(i)) & \emph{if } m=2i-2, \emph{ } i \geq 2\\
			0 & \emph{if } m<0
		\end{array}
		\right.$$
		of abelian groups.
	\end{proposition}
	
	\begin{proof}
		Let $p$ be the residue characteristic of the discrete valuation ring $\mathcal{O}_K$. The result with $p$-adic coefficients is \cite[Corollary~$2.16$]{antieau_K-theory_2024}. The result with $\Z[\tfrac{1}{p}]$-coefficients reduces to the case $n=1$, where the result follows from the description of the (classical) motivic cohomology of finite fields. The integral result is then a consequence of Lemma~\ref{lemmafinitechainringsdecomposition}.
	\end{proof}
	
	\begin{theorem}[Motivic cohomology of finite chain rings, after \cite{antieau_K-theory_2024}]
		Let $\mathcal{O}_K$ be a discrete valuation ring of mixed characteristic $(0,p)$ and with finite residue field $\F_q$, $\pi$ be a uniformizer of $\mathcal{O}_K$, and $n \geq 1$ be an integer. Then for every integer $i \geq 4p^n$,\footnote{Note that this is not an optimal lower bound on the integer $i$. See \cite[Theorem~$1.4$]{antieau_K-theory_2024} for a more precise result, in terms of the ramification index of $\mathcal{O}_K$.} the motivic complex $$\Z(i)^{\emph{mot}}(\mathcal{O}_K/\pi^n) \in \mathcal{D}(\Z)$$ is concentrated in degree one, where it is given by a group of order $(q^i-1)q^{i(n-1)}$.
	\end{theorem}
	
	\begin{proof}
		This is a consequence of Lemma~\ref{lemmafinitechainringsdecomposition}, the classical computation of the motivic cohomology of~$\F_q$, and \cite[Theorem~$1.4$ and Proposition~$1.5$]{antieau_K-theory_2024}.
	\end{proof}
	
	\begin{remark}[Nilpotence of $v_1$]\label{remarknilpotenceofv1}
		Antieau--Krause--Nikolaus also determine the nilpotence degree of the element $v_1$ in the mod $p$ syntomic cohomology of $\Z/p^n$ \cite[Theorem~$1.8$]{antieau_K-theory_2024}. This is a refinement of the key result in the study of $K(1)$-local $K$-theory of Bhatt--Clausen--Mathew \cite{bhatt_remarks_2020}. Note that this result on the nilpotence degree of $v_1$ can be reformulated, via Lemma~\ref{lemmafinitechainringsdecomposition}, as a statement on the mod $p$ motivic cohomology of $\Z/p^n$.  
	\end{remark}

    \section{Truncated polynomials}\label{Sectiontruncated}
	
	\vspace{-\parindent}
	\hspace{\parindent}
	
	In this section, we study the motivic cohomology of truncated polynomials, {\it i.e.}, the motivic cohomology of commutative rings of the form $R[x]/(x^e)$. Given a $\mathcal{D}(\Z)$-valued functor $F(-)$, a commutative ring $R$, and an integer $e \geq 1$, we use the notation
	$$F(R[x]/(x^e),(x)) := \text{fib}(F(R[x]/(x^e)) \longrightarrow F(R)),$$
	where the map is induced by the canonical projection $R[x]/(x^e) \rightarrow R$. 
	
	The relative $K$-theory $\text{K}(k[x]/(x^e),(x))$ of truncated polynomials over a perfect field $k$ of positive characteristic was computed by Hesselholt--Madsen \cite{hesselholt_K-theory_1997,hesselholt_cyclic_1997}, using topological restriction homology. Their calculation was reproved by Speirs \cite{speirs_K-theory_2020} using Nikolaus--Scholze's approach to topological cyclic homology \cite{nikolaus_topological_2018}, and by Mathew \cite{mathew_recent_2022} and Sulyma \cite{sulyma_floor_2023} using Bhatt--Morrow--Scholze's filtration on topological cyclic homology \cite{bhatt_topological_2019}. This last approach was then extended to mixed characteristic by Riggenbach \cite{riggenbach_K-theory_2025}. More precisely, Riggenbach used computations in prismatic cohomology to extend the previous result to a computation of the $p$-adic relative $K$-theory $\text{K}(R[x]/(x^e),(x);\Z_p)$ of perfectoid rings $R$, and also reproved the $p$-adic part of the known description of $\text{K}(\Z[x]/(x^e),(x))$, originally due to Angeltveit--Gerhardt--Hesselholt \cite{angeltveit_K-theory_2009}. 
	
	This recent progress would seem to indicate that $K$-theory calculations using equivariant stable homotopy may be pushed further by using cohomological techniques. Note however that the calculations in \cite{mathew_recent_2022,sulyma_floor_2023,riggenbach_K-theory_2025} are purely $p$-adic ones, as they rely on (instances of) prismatic cohomology. In fact, all of the previous integral calculations in mixed characteristic ({\it i.e.}, for $R$ the ring of integers of a number field) rely on a rational result of Soulé \cite{soule_rational_1980} and Staffeldt \cite{staffeldt_rational_1985}, who compute the ranks of the associated relative $K$-groups using equivariant homotopy theory. In this section, we revisit and extend this rational computation, and discuss some natural motivic refinements of the previous results.
	
	All of the above calculations use trace methods, via the Dundas--Goodwillie--McCarthy theorem. We first state the corresponding results at the level of cohomology theories.
	
	\begin{lemma}\label{lemmatruncatedpolynomialmotTC}
		Let $R$ be a commutative ring, and $e \geq 1$ be an integer. Then for every integer~$i \geq 0$, the natural map
		$$\Z(i)^{\emph{mot}}(R[x]/(x^e),(x)) \longrightarrow \Z(i)^{\emph{TC}}(R[x]/(x^e),(x))$$
		is an equivalence in the derived category $\mathcal{D}(\Z)$, where $\Z(i)^{\emph{TC}}$ denotes the $i^{\emph{th}}$ graded piece of the motivic filtration on integral topological cyclic homology $\emph{TC}$, as defined in \cite[Section~2.3]{bouis_motivic_2024}.
	\end{lemma}
	
	\begin{proof}
		This is a direct consequence of \cite[Remark~$3.21$]{bouis_motivic_2024}, and the fact that cdh sheaves are invariant under nilpotent extensions.
	\end{proof}
	
	\begin{corollary}\label{corollarytruncatedpolynomialpadicBMS}
		Let $R$ be a commutative ring, $e \geq 1$ be an integer, and $p$ be a prime number. Then for every integer $i \geq 0$, the natural map
		$$\Z_p(i)^{\emph{mot}}(R[x]/(x^e),(x)) \longrightarrow \Z_p(i)^{\emph{BMS}}(R[x]/(x^e),(x))$$
		is an equivalence in the derived category $\mathcal{D}(\Z_p)$, where $\Z_p(i)^{\emph{BMS}}$ denotes the weight-$i$ syntomic cohomology in the sense of \cite{bhatt_topological_2019}.
	\end{corollary}
	
	\begin{proof}
		This is a consequence of Lemma~\ref{lemmatruncatedpolynomialmotTC}.
	\end{proof}
	
	\begin{corollary}\label{corollarytruncatedpolynomialsmotdeRhamrational}
		Let $R$ be a commutative ring, and $e \geq 1$ be an integer. Then for every integer $i \geq 0$, there is a natural equivalence
		$$\Q(i)^{\emph{mot}}(R[x]/(x^e),(x)) \simeq \mathbb{L}\Omega^{<i}_{(R[x]/(x^e),(x))_{\Q}/\Q}[-1]$$
		in the derived category $\mathcal{D}(\Q)$.
	\end{corollary}
	
	\begin{proof}
            By \cite[Corollary~$4.66$]{bouis_motivic_2024}, and because cdh sheaves are invariant under nilpotent extensions, the natural map
            $$\Q(i)^{\text{mot}}(R[x]/(x^e),(x)) \longrightarrow \widehat{\mathbb{L}\Omega}^{\geq i}_{(R[x]/(x^e),(x))_{\Q}/\Q}$$
            is an equivalence in the derived category $\mathcal{D}(\Q)$. Moreover, there is a natural fibre sequence
            $$\widehat{\mathbb{L}\Omega}^{\geq i}_{(R[x]/(x^e),(x))_{\Q}/\Q} \longrightarrow \widehat{\mathbb{L}\Omega}_{(R[x]/(x^e),(x))_{\Q}/\Q} \longrightarrow \widehat{\mathbb{L}\Omega}^{<i}_{(R[x]/(x^e),(x))_{\Q}/\Q}$$
            in the derived category $\mathcal{D}(\Q)$, where the Hodge-completion on the last term can be removed, since only finitely many steps of the Hodge filtration appear in this truncated de Rham complex. Again using that cdh sheaves are invariant under nilpotent extensions, the desired result is then a consequence of cdh descent for the presheaf $\widehat{\mathbb{L}\Omega}_{-/\Q}$ on commutative $\Q$-algebras \cite[Lemma~$4.5$]{elmanto_motivic_2023}.
	\end{proof}
	
	\begin{lemma}
		For every commutative ring $R$ and integer $e \geq 1$, the object
		$$\Z(0)^{\emph{mot}}\big(R[x]/(x^e),(x)\big)$$ is zero
		in the derived category $\mathcal{D}(\Z)$.
	\end{lemma}
	
	\begin{proof}
		This is a consequence of the fact that the motivic complex $\Z(0)^{\text{mot}}$ is a cdh sheaf \cite[Example~$4.68$]{bouis_motivic_2024}.
	\end{proof}
	
	\begin{lemma}\label{lemmatruncatedpolynomiale-1forQ}
		For any integers $e \geq 1$ and $i \geq 0$, the complex
		$$\mathbb{L}\Omega^{\leq i}_{(\Q[x]/(x^e),(x))/\Q} \in \mathcal{D}(\Q)$$
		is concentrated in degree zero, given by a $\Q$-vector space of dimension $e-1$.
	\end{lemma}
	
	\begin{proof}
		This follows from a standard argument using the natural grading of the $\Q$-algebra $\Q[x]/(x^e)$ and the $\Q$-linear derivation $d : \Q[x]/(x^e) \rightarrow \Q[x]/(x^e)$ given by $d(x^j) = jx^j$; see for instance the proof of \cite[Proposition~$5$]{staffeldt_rational_1985}.
	\end{proof}
	
	\begin{theorem}\label{theoremtruncatedpolynomialmain}
		Let $R$ be a commutative ring such that the cotangent complex $\mathbb{L}_{(R\otimes_{\Z} \Q)/\Q}$ vanishes ({\it e.g.}, if $R \otimes_{\Z} \Q$ is ind-étale over $\Q$),\footnote{See \cite{mondal_ind-etale_2025} for more on this condition.} and $e \geq 1$ be an integer. Then for every integer~$i \geq 1$, there is a natural equivalence
		$$\Q(i)^{\emph{mot}}(R[x]/(x^e),(x)) \simeq (R \otimes_{\Z} \Q)^{e-1}[-1]$$
		in the derived category $\mathcal{D}(\Q)$.
	\end{theorem}
	
	\begin{proof}
		By Corollary~\ref{corollarytruncatedpolynomialsmotdeRhamrational}, there is a natural equivalence
		$$\Q(i)^{\text{mot}}(R[x]/(x^e),(x)) \simeq \mathbb{L}\Omega^{<i}_{(R[x]/(x^e),(x))_{\Q}/\Q}[-1]$$
		in the derived category $\mathcal{D}(\Q)$. By the Künneth formula for derived de Rham cohomology, and because all the positive powers of the cotangent complex $\mathbb{L}_{(R\otimes_{\Z} \Q)/\Q}$ vanish, there is a natural equivalence
		$$\mathbb{L}\Omega^{<i}_{(R[x]/(x^e),(x))_{\Q}/\Q} \simeq \mathbb{L}\Omega^{<i}_{(\Q[x]/(x^e),(x))/\Q} \otimes_{\Q} (R \otimes_{\Z} \Q)$$
		in the derived category $\mathcal{D}(\Q)$. The result is then a consequence of Lemma~\ref{lemmatruncatedpolynomiale-1forQ}.
	\end{proof}
	
	When $R$ is the ring of integers of a number field, the following result is due to Soulé \cite{soule_rational_1980} when $e=2$, and to Staffeldt \cite{staffeldt_rational_1985} for $e \geq 2$ a general integer. Their proof uses rational homotopy theory, and ultimately reduces to a computation in cyclic homology.
	
	\begin{corollary}\label{corollarytruncatedpolynomialrationalKtheory}
		Let $R$ be a commutative ring such that the cotangent complex $\mathbb{L}_{(R\otimes_{\Z} \Q)/\Q}$ vanishes, and $e \geq 1$ be an integer. Then for every integer $n \in \Z$, there is a natural isomorphism
		$$\emph{K}_n(R[x]/(x^e),(x);\Q) \cong \left\{
		\begin{array}{ll}
			(R \otimes_{\Z} \Q)^{e-1} & \emph{if } n \emph{ is odd and } n \geq 1 \\
			0 & \emph{otherwise}
		\end{array}
		\right.$$
		of abelian groups.
	\end{corollary}
	
	\begin{proof}
		This is a consequence of Theorem~\ref{theoremtruncatedpolynomialmain} and \cite[Corollary~$4.60$]{bouis_motivic_2024}.
	\end{proof}

        We now apply Theorem~\ref{theoremtruncatedpolynomialmain} to the case of the ring of integers of a number field, together with the $p$-adic computations of Riggenbach \cite{riggenbach_K-theory_2025}, to obtain the following integral result. Note that this result implies, via the Atiyah--Hirzebruch spectral sequence, an analogous computation for $K$-theory, which was first established, using independent techniques, by Angeltveit--Gerhardt--Hesselholt \cite{angeltveit_K-theory_2009}.

        \begin{corollary}\label{corollarytruncatedpolynomialintegral}
            Let $K$ be a number field, $\mathcal{O}_K$ be its ring of integers, and $e \geq 1$ be an integer. Then for any integers $i,n \geq 1$, there is an isomorphism
            $$\emph{H}^n_{\emph{mot}}(\mathcal{O}_K[x]/(x^e),(x),\Z(i)) \cong \left\{
		\begin{array}{llll}
			\mathcal{O}_K^{e-1} & \emph{if } n=1\\
			A_{i-1} & \emph{if } n=2\\
			0 & \emph{otherwise}
		\end{array}
		\right.$$
            of abelian groups, natural in $K$ for $n=1$, where $d$ is the degree of $K$ over $\Q$, $\mathfrak{D} \subseteq \mathcal{O}_K$ is the different ideal of $K$ over $\Q$, and $A_i$ is a finite group of order $((ei)!(i!)^{e-2})^d \times \vert \mathcal{O}_K/\mathfrak{D} \vert^{ei-i}.$ In particular, if $K=\Q$, then $A_i$ is a finite group of order $(ei!)(i!)^{e-2}$.
        \end{corollary}

        \begin{proof}
            By Lemma~\ref{lemmatruncatedpolynomialmotTC}, it suffices to compute the cohomology of the complex $$\Z(i)^{\text{TC}}(\mathcal{O}_K[x]/(x^e),(x)) \in \mathcal{D}(\Z).$$ By \cite[Corollary~$4.31$]{bouis_motivic_2024}, there is a natural cartesian square
            $$\begin{tikzcd}
			\Z(i)^{\text{TC}}(\mathcal{O}_K[x]/(x^e),(x)) \ar[r] \ar[d] & \widehat{\mathbb{L}\Omega}^{\geq i}_{(K[x]/(x^e),(x))/\Q} \ar[d] \\
			 \prod_{p \in \mathbb{P}} \Z_p(i)^{\text{BMS}}(\mathcal{O}_K[x]/(x^e),(x)) \ar[r] & \prod'_{p \in \mathbb{P}} \underline{\widehat{\mathbb{L}\Omega}}^{\geq i}_{(K^\wedge_p[x]/(x^e),(x))/\Q_p}
		\end{tikzcd}$$
            in the derived category $\mathcal{D}(\Z)$, where the left bottom corner is syntomic cohomology in the sense of \cite{bhatt_topological_2019} and the right bottom corner is rigid-analytic de Rham cohomology in the sense of \cite[Section~4.2]{bouis_motivic_2024}. 
            
            By \cite[Theorem~$6.4$ and proof of Theorem~$1.6$]{riggenbach_K-theory_2025}, there is a natural equivalence
            $$\Z_p(i)^{\text{BMS}}(\mathcal{O}_K[x]/(x^e),(x)) \simeq ((\mathcal{O}_K)^\wedge_p)^{e-1}[-1] \oplus A_{i-1,p}[-2]$$
            in the derived category $\mathcal{D}(\Z)$ for every prime number $p$, where $A_{i,p}$ is a finite group of order $p^{v_p(\vert A_i \vert)}$. By Theorem~\ref{theoremtruncatedpolynomialmain} and the proof of Corollary~\ref{corollarytruncatedpolynomialsmotdeRhamrational}, there is a natural equivalence
            $$\widehat{\mathbb{L}\Omega}^{\geq i}_{(K[x]/(x^e),(x))/\Q} \simeq K^{e-1}[-1]$$
            in the derived category $\mathcal{D}(\Z)$. Using the cdh descent result for rigid-analytic de Rham cohomology \cite[Corollary~$4.44$]{bouis_motivic_2024}, the same arguments as for the previous equivalence imply that there is a natural equivalence
            $$\prod_{p \in \mathbb{P}}{}' \,\underline{\widehat{\mathbb{L}\Omega}}^{\geq i}_{(K^\wedge_p[x]/(x^e),(x))/\Q_p} \simeq \Big(\prod_{p \in \mathbb{P}} (\mathcal{O}_K)^\wedge_p\Big)^{e-1}[-1]$$
            in the derived category $\mathcal{D}(\Z)$. The previous two equivalences are compatible by construction. Unwinding the prismatic computation \cite[Theorem~1.6]{riggenbach_K-theory_2025}, the only non-trivial compatibility to check between these equivalences is a consequence of the compatibility between the Nygaard filtration on (Breuil--Kisin twisted) absolute prismatic cohomology and the Hodge filtration on derived de Rham cohomology, which is \cite[Construction~5.5.3]{bhatt_absolute_2022}. The previous cartesian square then implies the desired result.
        \end{proof}
	
	We also deduce from the work of Riggenbach the following motivic interpretation of the analogous result in $K$-theory \cite[Theorem~$1.1$]{riggenbach_K-theory_2025}.
	
	\begin{theorem}[Truncated polynomials over perfectoids, after \cite{riggenbach_K-theory_2025}]
		Let $R$ be a perfectoid ring, and $e \geq 1$ be an integer. Then for every integer $i \geq 1$, there is a natural equivalence
		$$\Z_p(i)^{\emph{mot}}(R[x]/(x^e),(x)) \simeq \mathbb{W}_{ei}(R)/V_e\mathbb{W}_i(R)[-1]$$
		in the derived category $\mathcal{D}(\Z_p)$, where $\mathbb{W}(R)$ denotes the big Witt vectors of $R$, and $V$ the associated Verschiebung operator.
	\end{theorem}
	
	\begin{proof}
		This is a consequence of \cite[proof of Corollary~$6.5$]{riggenbach_K-theory_2025} and Corollary~\ref{corollarytruncatedpolynomialpadicBMS}.
	\end{proof}
	
	\begin{remark}[Cuspidal curves]
		The algebraic $K$-theory of cuspidal curves ({\it i.e.}, curves that are defined by an equation of the form $y^a-x^b$, for $a,b\geq 2$ coprime integers) was completely determined over a perfect $\F_p$-algebra by Hesselholt--Nikolaus \cite{hesselholt_algebraic_2020}, using Nikolaus--Scholze's approach \cite{nikolaus_topological_2018} to topological cyclic homology. This result was then generalised to mixed characteristic perfectoid rings by Riggenbach \cite{riggenbach_K-theory_2023}, ultimately relying on computations in relative topological Hochschild homology. It would seem that the associated Atiyah--Hirzebruch spectral sequence should degenerate in this context, thus providing a similar computation of the motivic cohomology of cuspidal curves. An interesting question would be whether these results can be reproved, or even extended to more general base rings, using techniques from prismatic cohomology and derived de Rham cohomology.
	\end{remark}

    \section{Perfect and semiperfect rings}\label{Sectionperfect}
	
	\vspace{-\parindent}
	\hspace{\parindent}
	
	Let $p$ be a prime number. It was proved by Kratzer \cite[Corollary~$5.5$]{kratzer_lambda_1980} that for every perfect $\F_p$-algebra $R$ and every integer $n \geq 1$, the $K$-group $\text{K}_n(R)$ is uniquely $p$-divisible (see also \cite{antieau_K-theory_2022} for a mixed characteristic generalisation). It was also proved by Kelly--Morrow that for every $\F_p$-algebra~$R$ with perfection $R_{\text{perf}}$, the natural map $\text{K}(R) \rightarrow \text{K}(R_{\text{perf}})$ is an equivalence after inverting $p$ (\cite[Lemma~$4.1$]{kelly_k-theory_2021}, see also \cite[Example~$2.1.11$]{elmanto_perfection_2020} and \cite[Theorem~$3.1.2$ and Proposition~$3.3.1$]{coulembier_K-theory_2023} for different proofs). The following result is a motivic refinement of these two facts. 
	
	\begin{theorem}[Motivic cohomology of perfect $\F_p$-schemes, after \cite{elmanto_motivic_2023}]\label{theoremmotiviccohomologyofperfectschemes}
		Let $X$ be a qcqs $\F_p$-scheme.
		\begin{enumerate}
			\item For every integer $i \geq 0$, the natural map
			$$\Z(i)^{\emph{mot}}(X)[\tfrac{1}{p}] \longrightarrow \Z(i)^{\emph{mot}}(X_{\emph{perf}})[\tfrac{1}{p}]$$
			is an equivalence in the derived category $\mathcal{D}(\Z[\tfrac{1}{p}])$.
			\item For every integer $i \geq 1$, the natural map
			$$\Z(i)^{\emph{mot}}(X_{\emph{perf}}) \longrightarrow \Z(i)^{\emph{mot}}(X_{\emph{perf}})[\tfrac{1}{p}]$$
			is an equivalence in the derived category $\mathcal{D}(\Z)$.
		\end{enumerate}
	\end{theorem}
	
	\begin{proof}
		By \cite[Theorem~$4.34$]{elmanto_motivic_2023},\footnote{This result is proved as a consequence of the same result in classical motivic cohomology \cite{geisser_k-theory_2000} and in syntomic cohomology \cite{antieau_beilinson_2020}, and ultimately goes back to the fact that the Frobenius acts by multiplication by $p^i$ on the logarithmic de Rham--Witt sheaf $W\Omega^i_{\text{log}}$.} for every integer $i \geq 0$, the natural map
		$$\phi_X^\ast : \Z(i)^{\text{mot}}(X) \longrightarrow \Z(i)^{\text{mot}}(X)$$
		induced by the absolute Frobenius $\phi_X : X \rightarrow X$ of a qcqs $\F_p$-scheme $X$ is multiplication by $p^i$. In particular, this natural map is an equivalence after inverting $p$, and $(1)$ is a consequence of this and the fact that the presheaf $\Z(i)^{\text{mot}}$ is finitary \cite[Theorem~$4.29\,(4)$]{elmanto_motivic_2023}. Similarly, the same result applied to the perfect $\F_p$-scheme $X_{\text{perf}}$ implies that multiplication by $p^i$ on the complex $\Z(i)^{\text{mot}}(X_{\text{perf}}) \in \mathcal{D}(\Z)$ is an equivalence. If $i \geq 1$, this is equivalent to the fact that the natural map
		$$\Z(i)^{\text{mot}}(X_{\text{perf}}) \longrightarrow \Z(i)^{\text{mot}}(X_{\text{perf}})[\tfrac{1}{p}]$$
		is an equivalence in the derived category $\mathcal{D}(\Z)$.
	\end{proof}
	
	\begin{remark}[Negative $K$-groups of perfect $\F_p$-algebras]\label{remarknegativeKgroupsperfectalgebras}
		It is possible to construct examples of perfect $\F_p$-algebras whose negative $K$-groups are not $p$-divisible \cite[Section~$3.3$]{coulembier_K-theory_2023}. Theorem~\ref{theoremmotiviccohomologyofperfectschemes}\,$(2)$ states that the only non-$p$-divisible information in the negative $K$-groups of a perfect $\F_p$-algebra $R$ actually come from weight zero motivic cohomology, {\it i.e.}, from the complex $R\Gamma_{\text{cdh}}(R,\Z)$ \cite[Example~$4.68$]{bouis_motivic_2024}.
	\end{remark}
	
	Recall that a $\F_p$-algebra is {\it semiperfect} if its Frobenius is surjective.
	
	\begin{corollary}[Motivic cohomology of semiperfect $\F_p$-algebras]
            Let $S$ be a semiperfect $\F_p$-algebra. Then for every integer $i \geq 1$, there is a natural fibre sequence
            $$\Z(i)^{\emph{mot}}(S) \longrightarrow \Z(i)^{\emph{mot}}(S_{\emph{perf}})[\tfrac{1}{p}] \longrightarrow (\Q_p/\Z_p)(i)^{\emph{syn}}(S)$$
            in the derived category $\mathcal{D}(\Z)$, where the last term $(\Q_p/\Z_p)(i)^{\emph{syn}}(S)$ denotes the cofibre of the canonical map $\Z_p(i)^{\emph{syn}}(S) \rightarrow \Q_p(i)^{\emph{syn}}(S)$.
	\end{corollary}
	
	\begin{proof}
            There is a natural fracture square
            $$\begin{tikzcd}
			\Z(i)^{\text{mot}}(S) \ar[r] \ar[d] & \Z(i)^{\text{mot}}(S)[\tfrac{1}{p}] \ar[d] \\
			 \Z_p(i)^{\text{mot}}(S) \ar[r] & \Q_p(i)^{\text{mot}}(S)
		\end{tikzcd}$$
            in the derived category $\mathcal{D}(\Z)$. By Theorem~\ref{theoremmotiviccohomologyofperfectschemes},(1), the right upper corner is naturally identified with $\Z(i)^{\text{mot}}(S_{\text{perf}})[\tfrac{1}{p}]$, so it suffices to identify the bottom horizontal map with the canonical map
            $$\Z_p(i)^{\text{syn}}(S) \longrightarrow \Q_p(i)^{\text{syn}}(S).$$
            By construction and derived Nakayama, this is equivalent to proving that the natural map
            $$\F_p(i)^{\text{mot}}(S) \longrightarrow \F_p(i)^{\text{syn}}(S)$$
            is an equivalence in the derived category $\mathcal{D}(\F_p)$. By \cite[Corollary~$4.40$]{elmanto_motivic_2023} (see also \cite[Theorem~$5.10$]{bouis_motivic_2024} for a mixed characteristic generalisation), this is in turn equivalent to the fact that
		$$R\Gamma_{\text{cdh}}(S,\widetilde{\nu}(i))[-i-1] \simeq 0$$
		in the derived category $\mathcal{D}(\F_p)$. By definition, the Frobenius map $\phi_S : S \rightarrow S$ is surjective, and has nilpotent kernel. The presheaf $R\Gamma_{\text{cdh}}(-,\widetilde{\nu}(i))[-i-1]$ is a finitary cdh sheaf, so the natural map
		$$R\Gamma_{\text{cdh}}(S,\widetilde{\nu}(i))[-i-1] \longrightarrow R\Gamma_{\text{cdh}}(S_{\text{perf}},\widetilde{\nu}(i))[-i-1]$$
		is then an equivalence in the derived category $\mathcal{D}(\F_p)$. The target of this map is zero by Theorem~\ref{theoremmotiviccohomologyofperfectschemes}\,$(2)$ (where we use that $i \geq 1$, and the same argument for syntomic cohomology), and applying \cite[Corollary~$4.40$]{elmanto_motivic_2023} to the perfect $\F_p$-algebra $S_{\text{perf}}$.
	\end{proof}

	\section{Valuation rings}\label{Sectionvaluation}
	
	\vspace{-\parindent}
	\hspace{\parindent}
	
	Recall that a valuation ring is an integral domain $V$ such that for any elements $f$ and $g$ in $V$,
	either $f \in gV$ or $g \in fV$. In recent years, valuation rings have been used as a way to bypass resolution of singularities, in order to adapt arguments from characteristic zero to more general contexts \cite{kerz_towards_2021,kelly_k-theory_2021,bouis_cartier_2023,bachmann_A^1-invariant_2025}. In this section, we describe the motivic cohomology of valuation rings (Theorems~\ref{theoremvaluationringslisse=mot} and~\ref{theoremvaluationringsmotiviccohomologyfinitecoefficients}). We start with the following result, stating that the motivic complexes~$\Z(i)^{\text{mot}}$, on henselian valuation rings, have a description purely in terms of algebraic cycles. See \cite[Section~9]{elmanto_motivic_2023} for related results over a field.
	
	\begin{theorem}\label{theoremvaluationringslisse=mot}
		Let $V$ be a henselian valuation ring. Then for every integer $i \geq 0$, the motivic complex $\Z(i)^{\emph{mot}}(V) \in \mathcal{D}(\Z)$ is in degrees at most $i$, and the lisse-motivic comparison map \cite[Definition~$2.1$]{bouis_weibel_2025}
		$$\Z(i)^{\emph{lisse}}(V) \longrightarrow \Z(i)^{\emph{mot}}(V)$$
		is an equivalence in the derived category $\mathcal{D}(\Z)$.
	\end{theorem}
	
	\begin{proof}
		The second statement already appears in the proof of \cite[Lemma~$3.25$]{bouis_weibel_2025}. As in \cite[Lemma~$3.25$]{bouis_weibel_2025} or \cite[Corollary~$2.12$]{bouis_weibel_2025}, the first statement is then a consequence of \cite[Corollary~$4.4$]{geisser_motivic_2004}.
	\end{proof}
	
	\begin{example}
		Let $V$ be a henselian valuation ring. By \cite[Example~$4.68$]{bouis_motivic_2024}, there is a natural equivalence
		$$\Z(0)^{\text{mot}}(V) \simeq \Z[0]$$
		in the derived category $\mathcal{D}(\Z)$. Similarly, Theorem~\ref{theoremvaluationringslisse=mot}, \cite[Example~$3.9$]{bouis_motivic_2024}, and the fact that the Picard group of a local ring is zero, imply that the motivic complex $\Z(1)^{\text{mot}}(V) \in \mathcal{D}(\Z)$ is concentrated in degree one, where it is given by
		$$\text{H}^1_{\text{mot}}(V,\Z(1)) \cong V^{\times}.$$
	\end{example}

	We now apply the results of the previous sections to give an alternative description of the motivic cohomology of valuation rings with finite coefficients. The following proposition will be used to reformulate the results of \cite{bouis_cartier_2023} on syntomic cohomology in terms of motivic cohomology.
	
	\begin{proposition}\label{propositionmotiviccohomologyvaluationringsmodpisindegreesatmosti}
		Let $p$ be a prime number, and $V$ be a henselian valuation ring. Then for any integers $i \geq 0$ and $k \geq 1$, there is a natural equivalence
		$$\Z/p^k(i)^{\emph{mot}}(V) \xlongrightarrow{\sim} \tau^{\leq i} \Z/p^k(i)^{\emph{syn}}(V)$$
		in the derived category $\mathcal{D}(\Z/p^k)$.
	\end{proposition}
	
	\begin{proof}
		Henselian valuation rings are local rings for the cdh topology, so this is a consequence of \cite[Theorem~$5.10$]{bouis_motivic_2024}.
	\end{proof}
	
	The following result is an analogue for valuation rings of Geisser--Levine's description of motivic cohomology of smooth $\F_p$-algebras \cite{geisser_k-theory_2000}. It can be deduced from the results of Kelly--Morrow \cite{kelly_k-theory_2021} and Elmanto--Morrow \cite{elmanto_motivic_2023}.
	
	\begin{theorem}\label{theoremmotiviccohomologyofcharpvaluationrings}
		Let $p$ be a prime number, and $V$ be a henselian valuation ring of characteristic $p$. Then for any integers $i \geq 0$ and $k \geq 1$, there is a natural equivalence
		$$\Z/p^k(i)^{\emph{mot}}(V) \xlongrightarrow{\sim} W_k\Omega^i_{V,\emph{log}}[-i]$$
		in the derived category $\mathcal{D}(\Z/p^k)$.
	\end{theorem}
	
	\begin{proof}
		Valuation rings of characteristic $p$ are Cartier smooth over $\F_p$ by results of Gabber--Ramero and Gabber \cite[Theorem~$3.4$]{bouis_cartier_2023}, so this is a consequence of \cite[Proposition~$5.1\,(ii)$]{luders_milnor_2023} and Proposition~\ref{propositionmotiviccohomologyvaluationringsmodpisindegreesatmosti}.
	\end{proof}
	
	We then prove a mixed characteristic version of Theorem~\ref{theoremmotiviccohomologyofcharpvaluationrings}, starting with the following $\ell$-adic general result.
	
	\begin{proposition}\label{propositionladicmotiviccohomologyofvaluationrings}
		Let $p$ be a prime number, and $V$ be a henselian valuation ring such that $p$ is invertible in $V$. Then for any integers $i \geq 0$ and $k \geq 1$, the Beilinson--Lichtenbaum comparison map \cite[Definition~$5.6$]{bouis_motivic_2024} naturally factors through an equivalence
		$$\Z/p^k(i)^{\emph{mot}}(V) \xlongrightarrow{\sim} \tau^{\leq i} R\Gamma_{\emph{ét}}(\emph{Spec}(V),\mu_{p^k}^{\otimes i})$$
		in the derived category $\mathcal{D}(\Z/p^k)$.
	\end{proposition}
	
	\begin{proof}
		By Proposition~\ref{propositionmotiviccohomologyvaluationringsmodpisindegreesatmosti}, the motivic complex $\Z/p^k(i)^{\text{mot}}(V) \in \mathcal{D}(\Z/p^k)$ is in degrees at most $i$, so the result is a consequence of \cite[Corollary~$5.6$]{bouis_motivic_2024}.
	\end{proof}
	
	The following result generalises Proposition~\ref{propositionladicmotiviccohomologyofvaluationrings} when $p$ is not necessarily invertible in the valuation ring $V$, at least over a perfectoid base. 
	
	\begin{theorem}[Motivic cohomology of valuation rings with finite coefficients]\label{theoremvaluationringsmotiviccohomologyfinitecoefficients}
		Let $p$ be a prime number, $V_0$ be a $p$-torsionfree valuation ring whose $p$\nobreakdash-completion is a perfectoid ring, and $V$ be a henselian valuation ring extension of $V_0$. Then for any integers $i \geq 0$ and $k \geq 1$, the Beilinson--Lichtenbaum comparison map \cite[Definition~$5.3$]{bouis_motivic_2024} induces a natural map
		$$\Z/p^k(i)^{\emph{mot}}(V) \longrightarrow \tau^{\leq i} R\Gamma_{\emph{ét}}(\emph{Spec}(V[\tfrac{1}{p}]),\mu_{p^k}^{\otimes i})$$
		in the derived category $\mathcal{D}(\Z/p^k)$, which is an isomorphism in degrees less than or equal to $i-1$. On~$\emph{H}^i$, this map is injective, with image generated by symbols, via the symbol map $$(V^\times)^{\otimes i} \rightarrow \emph{H}^i_{\emph{ét}}(\emph{Spec}(V[\tfrac{1}{p}]),\mu_{p^k}^{\otimes i}).$$
	\end{theorem}
	
	\begin{proof}
		The fact that the Beilinson--Lichtenbaum comparison map factors through the complex $$\tau^{\leq i} R\Gamma_{\text{ét}}(\text{Spec}(V[\tfrac{1}{p}]),\mu_{p^k}^{\otimes i}) \in \mathcal{D}(\Z/p^k)$$ is a consequence of Proposition~\ref{propositionmotiviccohomologyvaluationringsmodpisindegreesatmosti}. The isomorphism in degrees less than or equal to $i-1$ and the injectivity in degree $i$ of this map are then a consequence of \cite[Theorems~$3.1$ and~$4.12$]{bouis_cartier_2023}. The last statement is a consequence of the isomorphism
		$$\widehat{\text{K}}{}^{\text{M}}_i(V)/p^k \xlongrightarrow{\cong} \text{H}^i_{\text{mot}}(V,\Z/p^k(i))$$
		of abelian groups \cite[Theorem~$2.21$ and Corollary~$2.10$]{bouis_weibel_2025}.
	\end{proof}
	
	\begin{remark}
		The generation by symbols appearing in Theorem~\ref{theoremvaluationringsmotiviccohomologyfinitecoefficients} was also studied in the context of syntomic cohomology of general $p$-torsionfree $F$-smooth schemes by Bhatt--Mathew \cite{bhatt_syntomic_2023}. Note that all valuation rings are conjecturally $F$-smooth, and that the proof of Theorem~\ref{theoremvaluationringsmotiviccohomologyfinitecoefficients} adapts more generally to any henselian $F$-smooth valuation ring.
	\end{remark}

	\section{\texorpdfstring{$\C^{\star}$}{TEXT}-algebras}\label{SectionCstar}
	
	\vspace{-\parindent}
	\hspace{\parindent}
	
	By Gelfand representation theorem, the commutative $\C^\star$-algebras are exactly the algebras of continuous complex-valued functions $\mathscr{C}(X;\C)$ on a compact Hausdorff space $X$. An important theorem of Corti\~{n}as--Thom states that commutative $\C^\star$-algebras are $K$-regular \cite[Theorem~$1.5$]{cortinas_algebraic_2012}. This result was further generalised recently by Aoki to all smooth algebras over commutative $\C^\star$-algebras, and over a general local field \cite[Theorem~$8.7$]{aoki_k-theory_2026}. The following result is a motivic analogue of the latter result.
	
	\begin{theorem}[$\C^\star$-algebras are motivically regular, after \cite{cortinas_algebraic_2012,aoki_k-theory_2026}]
		Let $X$ be a compact Hausdorff space, $F$ be a characteristic zero local field, and $A$ be a smooth $\mathscr{C}(X;F)$\nobreakdash-algebra. Then for any integers $i \geq 0$ and $n \geq 0$, the natural map
		$$\Z(i)^{\emph{mot}}(A) \longrightarrow \Z(i)^{\emph{mot}}(A[T_1,\dots,T_n])$$
		is an equivalence in the derived category $\mathcal{D}(\Z)$.
	\end{theorem}
	
	\begin{proof}
		By \cite[Theorem~$8.7\,(2)$]{aoki_k-theory_2026}, the natural map
		$$\text{K}(A[T_1,\dots,T_n]) \longrightarrow \text{KH}(A[T_1,\dots,T_n])$$
		is an equivalence of spectra for every integer $n \geq 0$. By \cite[Remark~$3.27$ and Corollary~$4.60$]{bouis_motivic_2024}, this implies that the vertical maps in the commutative diagram
		$$\begin{tikzcd}
			\Z(i)^{\text{mot}}(A) \ar[r] \ar[d] & \Z(i)^{\text{mot}}(A[T_1,\dots,T_n]) \ar[d] \\
			\Z(i)^{\mathbb{A}^1}(A) \ar[r] & \Z(i)^{\mathbb{A}^1}(A[T_1,\dots,T_n])
		\end{tikzcd}$$
		are equivalences in the derived category $\mathcal{D}(\Z)$. The bottom horizontal map is an equivalence in the derived category $\mathcal{D}(\Z)$ by definition of the presheaf $\Z(i)^{\mathbb{A}^1}$ \cite{bachmann_A^1-invariant_2025}, see also \cite[Section~$6$]{bouis_motivic_2024}. So the top horizontal map is an equivalence in the derived category $\mathcal{D}(\Z)$.
	\end{proof}

	\bibliographystyle{alpha}
	
	{\footnotesize
\bibliography{biblio.bib}

@unpublished{bouis_weibel_2025,
	author = {Bouis, Tess},
	note = {\url{https://arxiv.org/abs/2507.00000}},
	title = {{Weibel vanishing and the projective bundle formula for mixed characteristic motivic cohomology}},
	year = {2025}}

@article {bloch_algebraic_1986,
    AUTHOR = {Bloch, Spencer},
     TITLE = {Algebraic cycles and higher {$K$}-theory},
   JOURNAL = {Adv. in Math.},
  FJOURNAL = {Advances in Mathematics},
    VOLUME = {61},
      YEAR = {1986},
    NUMBER = {3},
     PAGES = {267--304},
      ISSN = {0001-8708},
}

@unpublished{bhatt_absolute_2022,
	author = {Bhatt, Bhargav and Lurie, Jacob},
	note = {\url{https://arxiv.org/abs/2201.06120}},
	title = {Absolute prismatic cohomology},
	year = {2022}}

@book {voevodsky_cycles_2000,
    AUTHOR = {Voevodsky, Vladimir and Suslin, Andrei and Friedlander, Eric
              M.},
     TITLE = {Cycles, transfers, and motivic homology theories},
    SERIES = {Annals of Mathematics Studies},
    VOLUME = {143},
 PUBLISHER = {Princeton University Press, Princeton, NJ},
      YEAR = {2000},
     PAGES = {vi+254}
}

@inproceedings {quillen_higher_1973,
    AUTHOR = {Quillen, Daniel},
     TITLE = {Higher algebraic {$K$}-theory. {I}},
 BOOKTITLE = {Algebraic {$K$}-theory, {I}: {H}igher {$K$}-theories ({P}roc.
              {C}onf., {B}attelle {M}emorial {I}nst., {S}eattle, {W}ash.,
              1972)},
    SERIES = {Lecture Notes in Math., Vol. 341},
     PAGES = {85--147},
 PUBLISHER = {Springer, Berlin-New York},
      YEAR = {1973}
}

@inproceedings {lichtenbaum_values_1973,
    AUTHOR = {Lichtenbaum, Stephen},
     TITLE = {Values of zeta-functions, \'{e}tale cohomology, and algebraic
              {$K$}-theory},
 BOOKTITLE = {Algebraic {$K$}-theory, {II}: ``{C}lassical'' algebraic
              {$K$}-theory and connections with arithmetic ({P}roc. {C}onf.,
              {B}attelle {M}emorial {I}nst., {S}eattle, {W}ash., 1972)},
    SERIES = {Lecture Notes in Math., Vol. 342},
     PAGES = {489--501},
 PUBLISHER = {Springer, Berlin-New York},
      YEAR = {1973}
}

@incollection {beilinson_notes_1986,
    AUTHOR = {Be{\u{\i}}linson, Alexander A.},
     TITLE = {Notes on absolute {H}odge cohomology},
 BOOKTITLE = {Applications of algebraic {$K$}-theory to algebraic geometry
              and number theory, {P}art {I}, {II} ({B}oulder, {C}olo.,
              1983)},
    SERIES = {Contemp. Math.},
    VOLUME = {55},
     PAGES = {35--68},
 PUBLISHER = {Amer. Math. Soc., Providence, RI},
      YEAR = {1986}
}

@article {bhatt_syntomic_2023,
    AUTHOR = {Bhatt, Bhargav and Mathew, Akhil},
     TITLE = {Syntomic complexes and {$p$}-adic \'{e}tale {T}ate twists},
   JOURNAL = {Forum Math. Pi},
  FJOURNAL = {Forum of Mathematics. Pi},
    VOLUME = {11},
      YEAR = {2023},
     PAGES = {Paper No. e1, 26},
   MRCLASS = {14F30 (14F42)},
  MRNUMBER = {4530091},
       DOI = {10.1017/fmp.2022.21},
       URL = {https://doi.org/10.1017/fmp.2022.21},
}

@article {bhatt_prisms_2022,
    AUTHOR = {Bhatt, Bhargav and Scholze, Peter},
     TITLE = {Prisms and prismatic cohomology},
   JOURNAL = {Ann. of Math. (2)},
  FJOURNAL = {Annals of Mathematics. Second Series},
    VOLUME = {196},
      YEAR = {2022},
    NUMBER = {3},
     PAGES = {1135--1275}
}

@article {luders_milnor_2023,
    AUTHOR = {L\"{u}ders, Morten and Morrow, Matthew},
     TITLE = {Milnor {$K$}-theory of {$p$}-adic rings},
   JOURNAL = {J. Reine Angew. Math.},
  FJOURNAL = {Journal f\"{u}r die Reine und Angewandte Mathematik. [Crelle's
              Journal]},
    VOLUME = {796},
      YEAR = {2023},
     PAGES = {69--116},
      ISSN = {0075-4102}
}

@unpublished{bachmann_A^1-invariant_2025,
	author = {Bachmann, Tom and Elmanto, Elden and Morrow, Matthew},
	note = {\url{https://arxiv.org/abs/2508.09915}},
	title = {{{$\mathbb{A}^1$}-invariant motivic cohomology of schemes}},
        year = {2025},
	}

@article {bouis_cartier_2023,
    AUTHOR = {Bouis, Tess},
     TITLE = {Cartier smoothness in prismatic cohomology},
   JOURNAL = {J. Reine Angew. Math.},
  FJOURNAL = {Journal f\"{u}r die Reine und Angewandte Mathematik. [Crelle's
              Journal]},
    VOLUME = {805},
      YEAR = {2023},
     PAGES = {241--282}
}

@article {spitzweck_commutative_2018,
    AUTHOR = {Spitzweck, Markus},
     TITLE = {A commutative {$\mathbb P^1$}-spectrum representing motivic
              cohomology over {D}edekind domains},
   JOURNAL = {M\'{e}m. Soc. Math. Fr. (N.S.)},
  FJOURNAL = {M\'{e}moires de la Soci\'{e}t\'{e} Math\'{e}matique de France. Nouvelle S\'{e}rie},
    volume = {157},
      YEAR = {2018},
     PAGES = {110}
}

@article {aoki_k-theory_2026,
    AUTHOR = {Aoki, Ko},
     TITLE = {{{$K$}-theory of rings of continuous functions}},
   JOURNAL = {Camb. J. Math.},
  FJOURNAL = {Cambridge Journal of Mathematics},
      YEAR = {2024},
      note = {To appear},
}

@article {mondal_ind-etale_2025,
    AUTHOR = {Mondal, Shubhodip and Mukhopadhyay, Alapan},
     TITLE = {Ind-\'{e}tale versus formally \'{e}tale},
   JOURNAL = {Bull. Lond. Math. Soc.},
  FJOURNAL = {Bulletin of the London Mathematical Society},
    VOLUME = {57},
      YEAR = {2025},
    NUMBER = {4},
     PAGES = {1195--1207},
      ISSN = {0024-6093},
}

@article {riggenbach_K-theory_2025,
    AUTHOR = {Riggenbach, Noah},
     TITLE = {{$K$}-theory of truncated polynomials},
   JOURNAL = {Math. Z.},
  FJOURNAL = {Mathematische Zeitschrift},
    VOLUME = {310},
      YEAR = {2025},
    NUMBER = {3},
     PAGES = {Paper No. 59, 46},
}

@article {staffeldt_rational_1985,
    AUTHOR = {Staffeldt, Ross E.},
     TITLE = {Rational algebraic {$K$}-theory of certain truncated
              polynomial rings},
   JOURNAL = {Proc. Amer. Math. Soc.},
  FJOURNAL = {Proceedings of the American Mathematical Society},
    VOLUME = {95},
      YEAR = {1985},
    NUMBER = {2},
     PAGES = {191--198}
}

@inproceedings {soule_rational_1980,
    AUTHOR = {Soul\'{e}, Christophe},
     TITLE = {Rational {$K$}-theory of the dual numbers of a ring of
              algebraic integers},
 BOOKTITLE = {Algebraic {$K$}-theory, {E}vanston 1980 ({P}roc. {C}onf.,
              {N}orthwestern {U}niv., {E}vanston, {I}ll., 1980)},
    SERIES = {Lecture Notes in Math.},
    VOLUME = {854},
     PAGES = {402--408},
 PUBLISHER = {Springer, Berlin},
      YEAR = {1981}
}

@article {hesselholt_K-theory_1997,
    AUTHOR = {Hesselholt, Lars and Madsen, Ib},
     TITLE = {On the {$K$}-theory of finite algebras over {W}itt vectors of
              perfect fields},
   JOURNAL = {Topology},
  FJOURNAL = {Topology. An International Journal of Mathematics},
    VOLUME = {36},
      YEAR = {1997},
    NUMBER = {1},
     PAGES = {29--101},
      ISSN = {0040-9383}
}

@article {hesselholt_cyclic_1997,
    AUTHOR = {Hesselholt, Lars and Madsen, Ib},
     TITLE = {Cyclic polytopes and the {$K$}-theory of truncated polynomial
              algebras},
   JOURNAL = {Invent. Math.},
  FJOURNAL = {Inventiones Mathematicae},
    VOLUME = {130},
      YEAR = {1997},
    NUMBER = {1},
     PAGES = {73--97},
      ISSN = {0020-9910}
}

@article {speirs_K-theory_2020,
    AUTHOR = {Speirs, Martin},
     TITLE = {On the {$K$}-theory of truncated polynomial algebras,
              revisited},
   JOURNAL = {Adv. Math.},
  FJOURNAL = {Advances in Mathematics},
    VOLUME = {366},
      YEAR = {2020},
     PAGES = {107083, 18}
}

@article {sulyma_floor_2023,
    AUTHOR = {Sulyma, Yuri J. F.},
     TITLE = {Floor, ceiling, slopes, and {$K$}-theory},
   JOURNAL = {Ann. K-Theory},
  FJOURNAL = {Annals of K-Theory},
    VOLUME = {8},
      YEAR = {2023},
    NUMBER = {3},
     PAGES = {331--354}
}

@article {angeltveit_K-theory_2009,
    AUTHOR = {Angeltveit, Vigleik and Gerhardt, Teena and Hesselholt, Lars},
     TITLE = {On the {$K$}-theory of truncated polynomial algebras over the
              integers},
   JOURNAL = {J. Topol.},
  FJOURNAL = {Journal of Topology},
    VOLUME = {2},
      YEAR = {2009},
    NUMBER = {2},
     PAGES = {277--294}
}

@unpublished{kelly_procdh_2024,
	author = {Kelly, Shane and Saito, Shuji},
	note = {\url{https://arxiv.org/abs/2401.02699}},
	title = {{A procdh topology}},
	year = {2024}}

@article {cortinas_algebraic_2012,
    AUTHOR = {Cortiñas, Guillermo and Thom, Andreas},
     TITLE = {Algebraic geometry of topological spaces~{I}},
   JOURNAL = {Acta Math.},
  FJOURNAL = {Acta Mathematica},
    VOLUME = {209},
      YEAR = {2012},
    NUMBER = {1},
     PAGES = {83--131},
      ISSN = {0001-5962},
       DOI = {10.1007/s11511-012-0082-6},
       URL = {https://doi.org/10.1007/s11511-012-0082-6},
}

@article{mathew_recent_2022,
	author = {Mathew, Akhil},
	journal = {Bull. Lond. Math. Soc.},
	number = {1},
	pages = {1--44},
	title = {Some recent advances in topological {H}ochschild homology},
	volume = {54},
	year = {2022}}

@article{kerz_towards_2021,
	author = {Kerz, Moritz and Strunk, Florian and Tamme, Georg},
	journal = {Compos. Math.},
	number = {6},
	pages = {1143--1171},
	title = {Towards {V}orst's conjecture in positive characteristic},
	volume = {157},
	year = {2021}}

@article {antieau_K-theory_2022,
    AUTHOR = {Antieau, Benjamin and Mathew, Akhil and Morrow, Matthew},
     TITLE = {The {$K$}-theory of perfectoid rings},
   JOURNAL = {Doc. Math.},
  FJOURNAL = {Documenta Mathematica},
    VOLUME = {27},
      YEAR = {2022},
     PAGES = {1923--1952}
}

@article {bhatt_remarks_2020,
    AUTHOR = {Bhatt, Bhargav and Clausen, Dustin and Mathew, Akhil},
     TITLE = {Remarks on {$K (1)$}-local {$K$}\nobreakdash-theory},
   JOURNAL = {Selecta Math. (N.S.)},
  FJOURNAL = {Selecta Mathematica. New Series},
    VOLUME = {26},
      YEAR = {2020},
    NUMBER = {3},
     PAGES = {Paper No. 39, 16},
      ISSN = {1022-1824},
MRREVIEWER = {Bj\o rn Ian Dundas},
       DOI = {10.1007/s00029-020-00566-6},
       URL = {https://doi.org/10.1007/s00029-020-00566-6},
}

@article {geisser_motivic_2004,
    AUTHOR = {Geisser, Thomas},
     TITLE = {Motivic cohomology over {D}edekind rings},
   JOURNAL = {Math. Z.},
  FJOURNAL = {Mathematische Zeitschrift},
    VOLUME = {248},
      YEAR = {2004},
    NUMBER = {4},
     PAGES = {773--794},
      ISSN = {0025-5874},
MRREVIEWER = {Vladimir I. Guletski\u{\i}},
       DOI = {10.1007/s00209-004-0680-x},
       URL = {https://doi.org/10.1007/s00209-004-0680-x},
}

@article{kelly_k-theory_2021,
	author = {Kelly, Shane and Morrow, Matthew},
	journal = {Compos. Math.},
	number = {6},
	pages = {1121--1142},
	title = {{$K$}-theory of valuation rings},
	volume = {157},
	year = {2021}}

@article{nikolaus_topological_2018,
    author = {Nikolaus, Thomas and Scholze, Peter},
    journal = {Acta Math.},
    number = {2},
    pages = {203--409},
    title = {On topological cyclic homology},
    volume = {221},
    year = {2018}
}

@incollection {lichtenbaum_values_1984,
    AUTHOR = {Lichtenbaum, Stephen},
     TITLE = {Values of zeta-functions at nonnegative integers},
 BOOKTITLE = {Number theory, {N}oordwijkerhout 1983 ({N}oordwijkerhout,
              1983)},
    SERIES = {Lecture Notes in Math.},
    VOLUME = {1068},
     PAGES = {127--138},
 PUBLISHER = {Springer, Berlin},
      YEAR = {1984}
}

@article {levine_techniques_2001,
    AUTHOR = {Levine, Marc},
     TITLE = {Techniques of localization in the theory of algebraic cycles},
   JOURNAL = {J.~Algebraic Geom.},
  FJOURNAL = {Journal of Algebraic Geometry},
    VOLUME = {10},
      YEAR = {2001},
    NUMBER = {2},
     PAGES = {299--363},
      ISSN = {1056-3911},
   MRCLASS = {14C25 (14F42 19E15)},
  MRNUMBER = {1811558},
MRREVIEWER = {Tam\'{a}s Szamuely},
}

@article{geisser_k-theory_2000,
	author = {Geisser, Thomas and Levine, Marc},
	journal = {Invent. Math.},
	number = {3},
	pages = {459--493},
	title = {The {$K$}-theory of fields in characteristic {$p$}},
	volume = {139},
	year = {2000}}

@article {bachmann_very_2025,
    AUTHOR = {Bachmann, Tom},
     TITLE = {The very effective covers of {KO} and {KGL} over {D}edekind
              schemes},
   JOURNAL = {J. Eur. Math. Soc. (JEMS)},
  FJOURNAL = {Journal of the European Mathematical Society (JEMS)},
    VOLUME = {27},
      YEAR = {2025},
    NUMBER = {11},
     PAGES = {4705--4712},
      ISSN = {1435-9855,1435-9863},
}

@article{bhatt_topological_2019,
	author = {Bhatt, Bhargav and Morrow, Matthew and Scholze, Peter},
	journal = {Publ. Math. Inst. Hautes \'{E}tudes Sci.},
	pages = {199--310},
	title = {Topological {H}ochschild homology and integral {$p$}-adic {H}odge theory},
	volume = {129},
	year = {2019}}

@incollection {beilinson_height_1987,
    AUTHOR = {Be{\u{\i}}linson, Alexander A.},
     TITLE = {Height pairing between algebraic cycles},
 BOOKTITLE = {Current trends in arithmetical algebraic geometry ({A}rcata,
              {C}alif., 1985)},
    SERIES = {Contemp. Math.},
    VOLUME = {67},
     PAGES = {1--24},
 PUBLISHER = {Amer. Math. Soc., Providence, RI},
      YEAR = {1987}
}

@article {beilinson_notes_1987,
    AUTHOR = {Be{\u{\i}}linson, Alexander and MacPherson, Robert and Schechtman, Vadim},
     TITLE = {Notes on motivic cohomology},
   JOURNAL = {Duke Math. J.},
  FJOURNAL = {Duke Mathematical Journal},
    VOLUME = {54},
      YEAR = {1987},
    NUMBER = {2},
     PAGES = {679--710}
}

@article {antieau_beilinson_2020,
    AUTHOR = {Antieau, Benjamin and Mathew, Akhil and Morrow, Matthew and
              Nikolaus, Thomas},
     TITLE = {On the {B}eilinson fiber square},
   JOURNAL = {Duke Math. J.},
  FJOURNAL = {Duke Mathematical Journal},
    VOLUME = {171},
      YEAR = {2022},
    NUMBER = {18},
     PAGES = {3707--3806},
      ISSN = {0012-7094},
}

@article {kratzer_lambda_1980,
    AUTHOR = {Kratzer, Charles},
     TITLE = {{$\lambda $}-structure en {$K$}-th\'{e}orie alg\'{e}brique},
   JOURNAL = {Comment. Math. Helv.},
  FJOURNAL = {Commentarii Mathematici Helvetici},
    VOLUME = {55},
      YEAR = {1980},
    NUMBER = {2},
     PAGES = {233--254},
      ISSN = {0010-2571}
}

@unpublished{elmanto_motivic_2023,
	author = {Elmanto, Elden and Morrow, Matthew},
	note = {\url{https://arxiv.org/abs/2309.08463}},
	title = {{Motivic cohomology of equicharacteristic schemes}},
	year = {2023}}

@unpublished{coulembier_K-theory_2023,
	author = {Coulembier, Kevin},
	note = {\url{https://arxiv.org/abs/2304.01421}},
	title = {{{$K$}-theory and perfection}},
	year = {2023}}

@article {elmanto_perfection_2020,
    AUTHOR = {Elmanto, Elden and Khan, Adeel A.},
     TITLE = {Perfection in motivic homotopy theory},
   JOURNAL = {Proc. Lond. Math. Soc. (3)},
  FJOURNAL = {Proceedings of the London Mathematical Society. Third Series},
    VOLUME = {120},
      YEAR = {2020},
    NUMBER = {1},
     PAGES = {28--38}
}

@unpublished{antieau_K-theory_2024,
	author = {Antieau, Benjamin and Krause, Achim and Nikolaus, Thomas},
	note = {\url{https://arxiv.org/abs/2405.04329}},
	title = {{On the {$K$}-theory of {$\Z/p^n$}}},
	year = {2024}}

@unpublished{bouis_beilinson-lichtenbaum_2025,
	author = {Bouis, Tess and Kundu, Arnab},
	note = {\url{https://arxiv.org/abs/2506.09910}},
	title = {{Beilinson--Lichtenbaum phenomenon for motivic cohomology}},
	year = {2025}}

@unpublished{bouis_motivic_2024,
	author = {Bouis, Tess},
	note = {\url{https://arxiv.org/abs/2412.06635}},
	title = {{Motivic cohomology of mixed characteristic schemes}},
	year = {2024}}

@article {riggenbach_K-theory_2023,
    AUTHOR = {Riggenbach, Noah},
     TITLE = {{$K$}-theory of cuspidal curves over a perfectoid base and
              formal analogues},
   JOURNAL = {Adv. Math.},
  FJOURNAL = {Advances in Mathematics},
    VOLUME = {433},
      YEAR = {2023},
     PAGES = {Paper No. 109289, 40},
      ISSN = {0001-8708}
}

@incollection {hesselholt_algebraic_2020,
    AUTHOR = {Hesselholt, Lars and Nikolaus, Thomas},
 BOOKTITLE = {{$K$}-theory in algebra, analysis and topology},
    SERIES = {Contemp. Math.},
    VOLUME = {749},
TITLE = {Algebraic {$K$}-theory of planar cuspidal curves},
     PAGES = {139--148},
 PUBLISHER = {Amer. Math. Soc., Providence, RI},
      YEAR = {2020}
}
}

\end{document}